\documentclass[11pt]{article}
\usepackage{color}
\usepackage{epsfig}
\usepackage{subfig}
\usepackage{mathrsfs}
\usepackage{booktabs}
\usepackage{graphicx}
\usepackage{epstopdf}
\usepackage{multirow}
\usepackage{caption}
\usepackage{amsthm}
\usepackage{fancyhdr}
\usepackage{tabularx}
\usepackage{xfrac}
\usepackage{amsmath,amssymb}
\usepackage[ruled,linesnumbered]{algorithm2e}
\usepackage{indentfirst}
\newtheorem{theorem}{Theorem}[section]

\newtheorem{prop}{Proposition}[section]
\newtheorem{lemma}{Lemma}[section]
\newtheorem{remark}{Remark}[section]

\begin{document}

\title{Distributed Coverage Control of Multi-Agent Systems in Uncertain Environments using Heat Transfer Equations}
\author{Yinan Zheng, Chao Zhai \thanks{Yinan Zheng and Chao Zhai are with the School of Automation, China University of Geosciences, Wuhan 430074 P. R. China, and with Hubei Key Laboratory of Advanced Control and Intelligent Automation for Complex Systems, and Engineering Research Center of Intelligent Technology for Geo-Exploration, Ministry of Education, Wuhan 430074, P. R. China. Corresponding author: Chao Zhai~(email: zhaichao@amss.ac.cn).}}

\maketitle

\begin{abstract}
This paper addresses the coverage control problem of multi-agent system in the uncertain environment. With the aid of Voronoi partition, a distributed coverage control formulation of multi-agent system is proposed to complete the workload in the uncertain environments.
Driven by the gradient of thermal field, each agent is able to move around for clearing the workload on its own subregion.
Theoretical analysis is conducted to ensure the completion of workload in finite time. Finally, numerical simulations are carried out to demonstrate the effectiveness and advantages of the proposed coverage control approach as compared to other existing approaches.
\end{abstract}


\section{Introduction}
With the rapid development of communication and electronic technology, multi-agent systems have been widely used in many application fields, including exploration, surveillance, region coverage, search and rescue \cite{10}. Compared with a single agent, a team of agents are more efficient to carry out a complicated task, especially in the uncertain environment. Specifically, the workload of each individual in the system is relatively reduced and the increase in the number makes the whole system obtain more comprehensive information. Meanwhile, distributed system design enables each agent to cooperate to complete tasks while ensuring individual independence. Naturally, the research on system cooperation has become the main content, and has a lot of research results in \cite{1}-\cite{6}.

Among various coordination tasks, multi-agent region coverage is a hot research topic. Taking multi-sensor network as an example, the given area is covered to accomplish the global monitoring task, in which the deployment mode of the sensor has to be explored in order to achieve the maximum coverage rate and the shortest coverage time. To deal with this problem, distributed strategy (i.e., divide-and-conquer) may paly a crucial role in handling environmental uncertainties. For instance, \cite{7} and \cite{8} capitalize on Voronoi partition and introduce an optimization method to deploy the sensor networks. In \cite{9}, environment information is learned online with prior information. In \cite{10}-\cite{12}, the online workload partition and distributed coverage control algorithm are proposed, and the coverage time is estimated for algorithm comparison.

In \cite{13}, a centralized coverage control algorithm is proposed based on temperature transfer equation. While ensuring the coverage efficiency, it can also deal with the coverage problem of non-convex regions. However, agents are prone to collision during movement using this coverage method,
and there may be the problem that multiple agents cover the same area, which results in the degradation of efficiency of transitional coverage.
For this reason, this paper aims to propose a distributed coverage algorithm by taking into account workload partition and dynamic coverage.
In brief, key contributions of this work are summarized as follows:
\begin{enumerate}
	\item Propose a distributed coverage algorithm by integrating region partition (to handle the uncertainties) with coverage control (to complete the workload).
	\item Provide theoretical guarentee on the finite-time completion of coverage task with the aid of Delaunay triangulation.
    \item Make a quantitative comparison with other existing coverage control algorithms through numerical simulations.
\end{enumerate}

The remainder of the paper is organized as follows. Section 2 presents the formulation of coverage problem in uncertain environment.
Then the distributed coverage algorithm using heat transfer equation is proposed in Section 3, followed by theoretical analysis in Section 4.
Numerical simulations are carried out in Section 5. Finally, we draw conclusions and discuss future work in Section 6.

\section{Problem Formulation}

From the perspective of coverage time, the coverage task of multi-agent systems can be transformed into a single objective optimization problem.
In this section, we develop distributed region partition method of multi-agent system, introduce the related concepts of coverage problem, and come up with the optimization goal.

\subsection{Region partition}

Consider the 2-dimensional workspace $D~\subset~\mathbb{R}^{2}$ with Lipschitz continuous boundary for $N$ mobile agents to complete the coverage task, and the location of the agent $i$ can be represented by spatial coordinates $\bf{s}_i=(x_i,y_i)$. In addition, any point in the workspace is expressed as $\mathbf{x}=(x,y)$. The communication topology of multi-agent systems can be expressed as an undirected graph $G=(V,E)$, where $V$ represents a collection of $N$ nodes, and each node corresponds to an individual agent. Here $E$ represents the set of connections between nodes, that is, the interactive relationship between agents. For $i,j \in \{1,2,3,...,N\}$, let $w_{ij} \in \{0,1\}$ represent the communication or connection status between Agent $i$ and Agent $j$. When $w_{ij}=1$, it means that Agent $i$ is connected with Agent $j$. In this problem, we assume that there is communication relationship between agents with overlapping edges in the subregion after Voronoi partition in \cite{14}. For the agent $i$, its Voronoi cell $D_i$ is defined as
\begin{equation} \label{1}
	D_i=\{\mathbf{x}\in D | d(\mathbf{x},\bf{s}_i)\le d(\mathbf{x},\mathbf{s}_j),\forall~j\in\{1,2,3,...,N\},i\ne j\},
\end{equation}
where $d(\mathbf{x},\bf{s}_i)=||\mathbf{x}-\bf{s}_i||$ is the Euclidean distance between the location of $\mathbf{x}$ and $\bf{s}_i$.
The partition method is shown in Figure 1.

\begin{figure}[h]
	\centering
	\includegraphics[width=\linewidth]{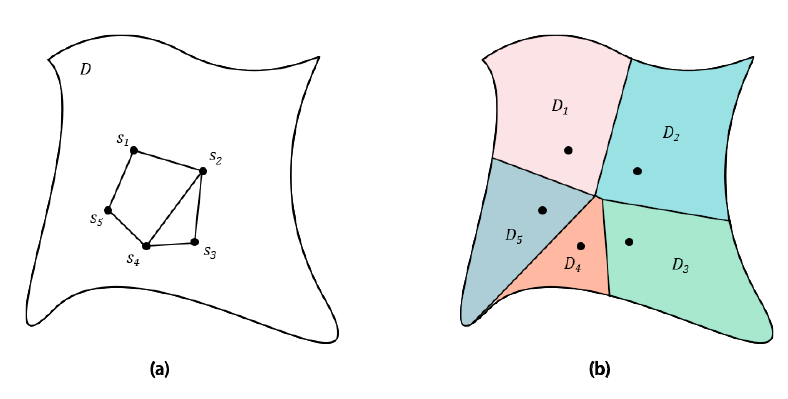}
	\caption{Voronoi partition method on the region $D$. (a): $\bf{s}_i$ represents the location of agent $i$, and the communication  topology of multi-agent systems can be expressed as the graph. (b): The subregion configuration after Voronoi partition.}
\end{figure}

With the assistance of Voronoi partition, the whole area can be divided into $N$ subregions with overlapping boundaries, and a single agent (say $i$) only works in its own Voronoi cell $D_i$. The relatively independent workspace enables the distributed implementation of coverage algorithm, which has the following benefits. First of all, the distance from any point in the Voronoi cell to its agent is always less than that to other agents, which is determined by the nature of Voronoi diagram, so that the agent only needs to be responsible for its nearest area to speed up coverage efficiency. Secondly, because of region partition, the agent only needs to know the work tasks in partial areas, which is consistent with most application scenarios, and relatively separated working region can avoid the collision among agents. This distributed design makes the system more robust. Finally, note that each agent moves along the gradient of potential energy field in \cite{13}. For multiple agents close to each other, when the gradient is extremely small, it may eventually lead to the problem that multiple agents pile up and stagnate, which may lead to redundant coverage. Voronoi partition can effectively alleviate this problem.  However, the work interval divided by distance can not guarantee the same workload, which will lead to unequal coverage time of agents and the wasting of coverage resources. The solutions are thus introduced in later sections.

\subsection{Coverage metrics}
After dividing the coverage region into multiple subregions, we need to determine the performance index, which is used to describe the coverage level of agents on the responsible region. Due to hardware limitations of agents, the coverage scope of each agent is limited at certain time.
\cite{15}~proposed a monotonically decreasing differentiable function $p_{i}(\mathbf{x},\bf{s}_i)$ to quantify the coverage capacity of agent
$i$ at the point $\mathbf{x}$ per unit time.
\begin{equation} \label{2}
	p_{i}(\mathbf{x},\bf{s}_i)=
	\begin{cases}
		P_{i}e^{-\lambda_{i}d(\mathbf{x},\bf{s}_i)}&d(\mathbf{x},\bf{s}_i) \leq r \\
		0& d(\mathbf{x},\bf{s}_i) > r
	\end{cases},
\end{equation}
where $P_{i}>0$ represents the maximum coverage capacity of the agent $i$, and $\lambda_i$ is the attenuation coefficient and $r$ is the effective coverage radius.
In the actual coverage process, the agent does not remain stationary. Assume that the coverage capacity of each agent remains unchanged. For agent $i$, its total coverage to point $\mathbf{x}\in D$ during the period $[0,t]$ is given by
\begin{equation} \label{3}
	c_{i}(\textbf{x},t)=\int_{0}^{t} p_{i}(\textbf{x},s_{i}(\tau))d\tau.
\end{equation}
For the region $D$, the workload is represented by $m(\textbf{x},t)$, which can be obtained as follows
\begin{equation} \label{4}
	m(\textbf{x},t) = \max\left(0,m_0(\textbf{x})-\sum_{i=1}^{N}c_i(\textbf{x},t)\right),
\end{equation}
where $m_0(\textbf{x})$ is the initial value of workload on coverage region. Thus, it can be expressed that the total workload of the area is
\begin{equation} \label{5}
	M(t) = \int_{D}^{}m(\textbf{x},t)d\textbf{x},\quad M(0) = \int_{D}^{}m_0(\textbf{x})d\textbf{x}.
\end{equation}
Through the above analysis, we can transform the task of area coverage into making the total workload equal to 0, i.e. $\lim_{t \to \infty} M(t)=0$. However, in addition to completing the coverage task, it is also required to achieve region coverage in the shortest time.
Let $T^*$ represent the optimal time of covering the whole region and $T$ is the actual coverage time. Note that the coverage velocity of each agent is equal, and it is denoted by $v$. In light of (\ref{2}), $v$ can be computed by
\begin{equation} \label{6}
	v=\int_{d(\textbf{x},\bf{s}_i) \leq r}^{} p_{i}(\textbf{x},s_{i})d\textbf{x}.
\end{equation}
$T^*$ can be computed by
\begin{equation} \label{7}
	T^*=\frac{1}{Nv}M(0),
\end{equation}
which is unavailable due to environmental uncertainties, and the error is given by $\Delta T=T-T^*$. The objective of this study is to design a distributed online coverage scheme to minimize the actual coverage time.

\section{Coverage Control Algorithm}
In this section, we present a coverage control method for multi-agent systems to clear the workload in a given region by using heat transfer equation.

\subsection{Heat field-induced control}
Instead of generating global temperature field~\cite{13}, the coverage of each Voronoi cell $D_i$ is taken into account by solving heat transfer equation locally. While adopting the distributed method, it also retains the global advantage of temperature field on the coverage area. The coverage for a certain time can be regarded as a static temperature field with adiabatic boundary.
For region $D_i$, its temperature field $T_i$ at time $t$ satisfies
\begin{equation} \label{8}
	\alpha\Delta T_i(\textbf{x},t)+h_i(\textbf{x},t)=\beta T_i(\textbf{x},t),
\end{equation}
with the Neumann boundary condition
\begin{equation} \label{9}
	\frac{\partial T_i(\textbf{x},t)}{\partial \mathbf{n} }=0,
\end{equation}
where $\alpha>0$, $\beta>0$, $\Delta$ is a Laplace operator, $\mathbf{n}$ is the outward normal vector. Note that the above equation is independent of time, and the boundary of subregions results from Voronoi partition. In the heat equation (\ref{8}), $\alpha\Delta T_i$ represents the net heat flow. Thermal diffusivity $\alpha$ represents the ability of the temperature of all parts of the object, and it tends to be the same when the object is heated or cooled. The larger the value, the smaller the temperature difference in all parts of the object. To deal with the coverage problem, we convert the workload level into internal heat source $h_i(\textbf{x},t)$
\begin{equation} \label{10}
	h_{i}(\textbf{x},t)=\frac{m(\textbf{x},t)}{\bar{m}_i}, \quad x \in D_i,
\end{equation}
where $\bar{m}_i$ represents the maximum workload in the initial case for region $D_i$.
\begin{equation} \label{11}
	\bar{m}_i =  \max(m_0(\textbf{x})).
\end{equation}
When the agent does not cover a point $\textbf{x}$ in the area, the energy of heat source releases heat flow to the low-temperature area. When the agent completes the coverage of point $\textbf{x}$, it will no longer provide energy as a heat source. Similar to \cite{13}, the term $\beta T_i$ governs cooling over the whole region. It can be considered as the difference between the current temperature and $0$ degree. The larger the difference is, the more remarkable the cooling effect is. Parameter $\beta$ is the global cooling coefficient, which controls the global cooling effect.

According to Fourier theorem, the heat flow vector at any point in the system is directly proportional to the temperature gradient at that point, but in the opposite direction, always along the direction of temperature reduction. The motion of the agent can be compared with the flow of heat flow, but the motion direction is opposite to the heat flow vector, which makes the agent always move in the direction of heat source, that is, in the direction towards more workload. Therefore, the dynamics of agent $i$ is given as
\begin{equation} \label{12}
	\bf{\dot{s}_i}=u_{i},
\end{equation}
where the control input $\bf{u}_i$ is given by
\begin{equation} \label{13}
	\bf{u}_{i}=V_i\frac{\bigtriangledown T_i(\bf{s}_i(t),t) }{||\bigtriangledown T_i(\bf{s}_i(t),t) ||},
\end{equation}
with the constant speed $V_i$.

\begin{algorithm}[t!]
	\caption{Coverage Control on General Subregions}\label{algorithm1}
	Partition the region $D$ with (\ref{1}) \;
	\tcp{Agent $i$ performs simultaneously as follow }
	\While{$M(t)>0$}{
			Calculate the temperature field on $D_i$ with (\ref{8})\;
			Compute control input $u_i$ with (\ref{13})\;
			Update the total workload $M(t)$ with (\ref{5})\;}
\end{algorithm}

\subsection{Implementation of cooperative coverage algorithm}
Section 2.1 presents a method of workspace partition based on Voronoi diagram, but there is a problem left, that is, even if Voronoi partition can ensure the minimum distance from the point in the region to the agent, it can not ensure the equal workload of each workspace. In extreme cases, the time required for agents to cover their respective areas varies greatly. This may result in the fact that some agents complete their workload in advance and are in the idle condition, which causes the wasting of resources and decrease of efficiency. In this section, we propose a method of iterative partition, which can make full use of agent resources while ensuring the division of subregions, and finally complete the overall coverage mission in finite time.

Let $t$ denote the time from the start of the coverage to the current time, and $t'$ represents the time from the start of iteration to the current time. Suppose that in the $k$-th iteration, agent $i$ first completes the coverage task of its subregion $D^k_i$, it takes total time $T^k$. We need to determine the segmentation result of iteration $k+1$. Before partition, we need to update the workload $m_0^{k+1}(\textbf{x})$. The workload at the end of last iteration is used as the initial value of the next iteration.
\begin{equation} \label{14}
	m_0^{k+1}(\textbf{x})=m^{k}(\textbf{x},T^k),
\end{equation}
where $m^{k}(\textbf{x},T^k)$ is the value of function $m^{k}(\textbf{x},t')$ at time $T^k$. Similar to (\ref{4}), $m^{k}(\textbf{x},t')$ is defined as
\begin{equation} \label{15}
	m^{k}(\textbf{x},t')=\max\left(0,m_0^{k}(\textbf{x})-\sum_{i=1}^{N} \int_{0}^{t'} p_{i}(\textbf{x},s_{i}^k(\tau))d\tau\right).
\end{equation}
Then agent $i$ moves to the subregion of agent $j$ that it can communicate with. Besides, the remaining workload in $D_j^k$ is the largest compared with that of other neighbors of agent $i$. Thus, the number $j$ can be determined by
\begin{equation} \label{16}
	j=\arg \underbrace{\max }_{j} w_{ij}\int_{D_j^k}^{} m_{0}^{k+1}(\textbf{x})d\textbf{x}.
\end{equation}
Set the initial position of agent $i$ in the $k+1$ iteration as the point with the largest workload in area $D_j^k$, and assume that the moving time is ignored, because the moving speed of agent without coverage can be much greater than $V_i$, that is, the time consumption between iterations is ignored
\begin{equation} \label{17}
	\bf{s}_i^{k+1}(0)=\arg \underbrace{\max }_{\textbf{x}} m_0^{k+1}(\textbf{x}), \quad \textbf{x} \in D_j^k.
\end{equation}
After that, we can use the (\ref{1}) to get new region partition results $D_i^{k+1}$. Since the solution of heat transfer equation in Section 3.1
is not directly related to time, it can be used to control the movement of agent, but the difference is that when any agent completes the subregion, it will be regarded as the end of iterations and go to the next iteration until $M(t)$ approaches 0.
In order to achieve finite-time coverage, we propose a temperature field update method (i.e., Maximal Update). In each iteration, Maximal Update is used after Real-time Update. Generally, the time $T_u^k$ of Real-time Update can be set for each iteration. After that, when $t'> T_u^k$, Maximal Update is adopted until the end of this iteration.

When the temperature field remains unchanged, the agent moves towards local maximum point according to (\ref{13}). For the subregion where each agent is located, Delaunay triangulation in \cite{16} can be carried out to obtain the triangular mesh of the region. For the triangular region with local maximum points $(x_{max}, y_{max})$, there is a minimum coverage circle centered on local maximum point, which can cover the triangular region, and it satisfies
\begin{equation} \label{18}
(x-x_{max})^2+(y-y_{max})^2=r_{min}^2, \quad r_{min}\leq H_{max} \leq <r.
\end{equation}
Here $r_{min}$ is the radius of the minimum coverage circle, which is equivalent to the maximum distance from the local maximum point to the vertex of the triangle. $H_{max}$ is the maximum mesh edge mesh. At the same time, it is required that the coverage radius of the agent should be greater than $H_{max}$, which ensures that once the agent enters the triangular area, the whole area will be within the coverage range.
When the agent is moving towards the local maximum point, the temperature field remains unchanged. When the agent enters the triangular area where the local maximum value is located, it is ensured that the coverage of maximum point by defining the coverage circle, As shown in the Figure 2. At this time, we keep the agent stationary and update the temperature field after completing the work task of the triangular area. The agent runs to the next local maximum point according to the gradient of temperature field. This iteration continues until the end. (For details, refer to Algorithm~2).
Due to less workload on subregions at the end of coverage, we set the maximum iterations $k_{max}$. When $k=k_{max}$, the agent is only responsible for clearing workload on its subregion until all workload is completed. The complete coverage algorithm is presneted in Algorithm 3.
\begin{figure}[t!]
	\centering
	\includegraphics[width=12cm]{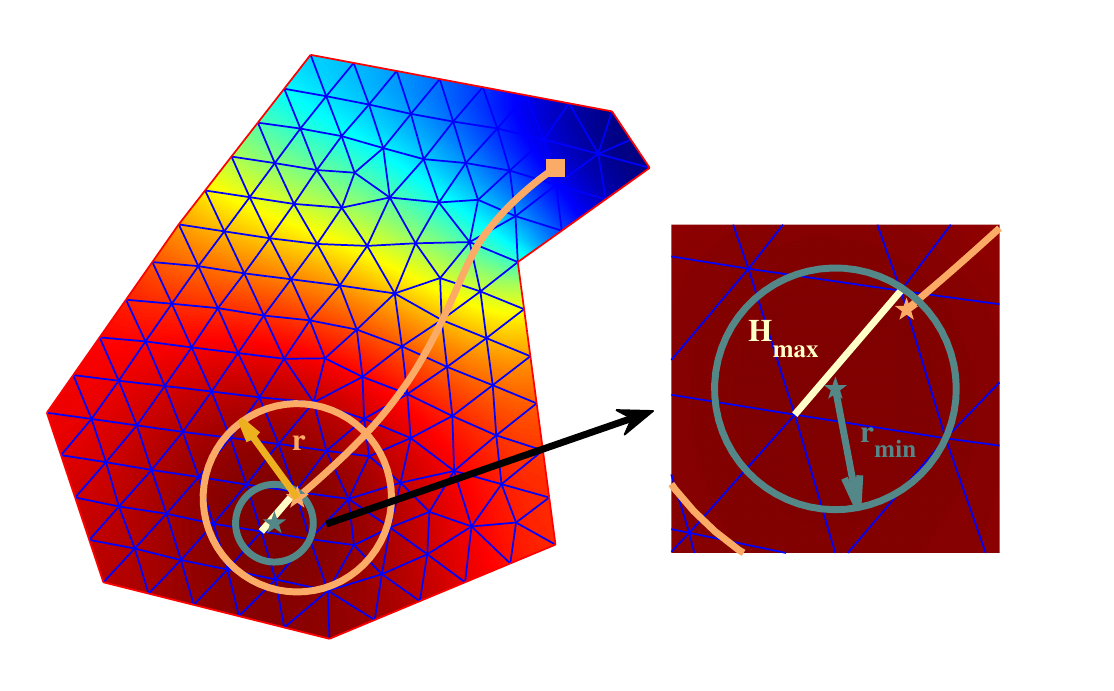}
	\caption{Illustration of Maximal Update Method: blackish green pentagram is the local maximum point and its minimun coverage circle. The trajectory of agent is shown in the figure and the pentagram is the location that agent reach the triangular area. In addition, the figure shows the length relationship between $r$, $r_{min}$, $H_{max}$.}
\end{figure}


\begin{algorithm}[t]
	\caption{Coverage Control on Triangulated Subregions}\label{algorithm2}
	Partition the region $D$ with (\ref{1}) \;
	\tcp{Agent $i$ performs simultaneously as follow }
	\While{$\int_{D_i^k}^{} m^{k}(\textbf{x},t')d\textbf{x}>0$ and no other agents complete workload}{
		\eIf{$t'\leq T_u^k$}{Calculate the temperature field on subregion $D_i^k$ with (\ref{8})\;
		Update $s_i$ with (\ref{12}) and (\ref{13}) and $m^{k}(\textbf{x},t')$ with (\ref{15})\;}
		{
		\eIf{agent $i$ is at the boundary of triangulation region with local maximum}
		{Keep static and clear triangle workload\;
		Calculate the temperature field on subregion $D_i^k$ with (\ref{8})\;
		Update $s_i$ with (\ref{12}) and (\ref{13}) and  $m^{k}(\textbf{x},t')$ with (\ref{15})\;}
		{
		Update $s_i$ with (\ref{12}) and (\ref{13}) and  $m^{k}(\textbf{x},t')$ with (\ref{15})\;}}
	}
\end{algorithm}

\begin{algorithm}[t]
\caption{Distributed Coverage Control on Global Region}\label{algorithm3}
$k \leftarrow 1$\;
\While{$M(t)>0$}{
	\eIf{$k<k_{max}$}{
		Run Algorithm 2 \;
		Determine the subregion and move to $D_j^k$ with (\ref{16})\;
		Drive agent to the maximum workload with (\ref{17})\;
        $k \leftarrow k+1$\;}
{Run Algorithm 1\;}}
\end{algorithm}

\section{Theoretical Analysis}

In this section, we provide theoretical results on the finite-time coverage of multi-agent systems using the proposed thermal field-induced coverage control method.

\begin{prop}
$\forall \Delta t \geq0$, $\textbf{x} \in D_i$, $i\in \textsc{Z}^{+}$, $h_i(\textbf{x},t)=0 \Longleftrightarrow\bigtriangledown T_i(\textbf{x},t+\Delta t)=0$.
\end{prop}

\begin{proof}
For $h_i(\textbf{x},t)=0 \Rightarrow \bigtriangledown T_i(\textbf{x},t)=0$ and $\bigtriangledown T_i(\textbf{x},t+\Delta t)=0$. When $h_i(\textbf{x},t)=0$,  $\textbf{x} \in D_i$, rewrite the (\ref{8}) as follows:
$$
	\alpha(\frac{\partial^2 T_i(\textbf{x},t)}{\partial x^2}+\frac{\partial^2 T_i(\textbf{x},t)}{\partial y^2})=\beta T_i(\textbf{x},t).
$$
Obviously, $T(\textbf{x},t)=0$ is the solution of the partial differential equation and satisfies the boundary conditions at the same time. From which we can deduce
$$
\frac{\partial T_i(\textbf{x},t)}{\partial x}=0, \quad \frac{\partial T_i(\textbf{x},t)}{\partial y}=0.
$$
Because $h$ is a monotonically decreasing function with $t$ and $h \geq 0$, we can get $h_i(\textbf{x},t+\Delta t)=0$ easily.
Likewise, we can obtain $\bigtriangledown T_i(\textbf{x},t+\Delta t)=0$. The conditions $\bigtriangledown T_i(\textbf{x},t)=0$ and $\bigtriangledown T_i(\textbf{x},t+\Delta t)=0$ leads to  $h_i(\textbf{x},t)=0$.
From
$$
\bigtriangledown T_i(\textbf{x},t)=0 \Rightarrow \Delta T_i(\textbf{x},t)=0, \quad \bigtriangledown T_i(\textbf{x},t+\Delta t)=0 \Rightarrow \Delta T_i(\textbf{x},t+\Delta t)=0,
$$
we can get
$$
h_i(\textbf{x},t)=\beta T_i(\textbf{x},t)=c_1, \quad h_i(\textbf{x},t+\Delta t)=\beta T_i(\textbf{x},t+\Delta t)=c_2,
$$
where $c_1$ and $c_2$ are constants. Owing to $c_1 \neq 0$, we have
$$
h_i(\textbf{x},t+\Delta t)=\frac{m(\textbf{x},t+\Delta t)}{\hat{m}_i},
$$
where $\frac{m(\textbf{x},t)}{\bar{m}_i}=c_1$. However, $m(\textbf{x},t)$ is a monotonically increasing function and $c_1 \neq 0$, leading
to $\frac{m(\textbf{x},t+\Delta t)}{\bar{m}_i} \neq c_1$, which conflicts with $h_i(\textbf{x},t+\Delta t)=c_2$, even $c_2=c_1$. According to
the above counter example, it is concluded that $c_1=0$. Thus, the proof is completed.
\end{proof}

\begin{remark}
$\forall \Delta t >0$, $\textbf{x} \in D_i$. When $h_i(\textbf{x},t)=0$, the coverage of region $D_i$ is finished. $\bigtriangledown T_i(\textbf{x},t)=0$ and $\bigtriangledown T_i(\textbf{x},t+\Delta t)=0$ means the gradient of temperature field is zero everywhere, that is, the agent stops moving. It can be concluded that when the gradient of temperature field becomes zero at a certain time and then all become zero (the agent stops moving at a certain time and then remains stationary), the coverage of the area is completed.
\end{remark}

\begin{lemma}\label{maximum has value}
$\forall x_0\in D_i$, $i\in \textsc{Z}^{+}$, $H(T_i(x_0,t))<0\Longrightarrow h_i(x_0,t)>0$,
where $H(T_i(x_0,t))<0$ is Hessian matrix of $T_i(x_0,t)$.
\end{lemma}
\begin{proof}
In the 2D plane, we can get Hessian matrix of $T_i(x_0,t)$.
$$
	H(T_i(x_0,t))=\begin{bmatrix}
		 \frac{\partial^2 T_i(x_0,t)}{\partial x\partial x} & \frac{\partial^2 T_i(x_0,t)}{\partial x\partial y}\\
		\frac{\partial^2 T_i(x_0,t)}{\partial y\partial x}&\frac{\partial^2 T_i(x_0,t)}{\partial y\partial y}
	\end{bmatrix}.
$$
Since $H(T_i(x_0,t))$ is negative definite, we can get $\frac{\partial^2 T_i(x_0,t)}{\partial x\partial x}<0$ and $\frac{\partial^2 T_i(x_0,t)}{\partial y\partial y}<0$. It is concluded that
$$
	\alpha\Delta T_i(\textbf{x},t) = \alpha(\frac{\partial^2 T_i(x_0,t)}{\partial x\partial x}+\frac{\partial^2 T_i(x_0,t)}{\partial y\partial y})<0.
$$
Because of the existence of the term $\beta T_i(\textbf{x},t)$ in the equation, the temperature of whole field is greater than zero. The specific reasons can be seen in the physical principle of heat transfer equation and (\ref{10}). So, we can get
$$
T_i(x_0,t)>0.
$$
For point $x_0$, (\ref{8}) can be written as
$$
\alpha\Delta T_i(x_0,t)+h_i(x_0,t)=\beta T_i(x_0,t).
$$
Through the above analysis, it is concluded that the right side of the equation is less than 0 and $\alpha\Delta T_i(x_0,t)$ is less than 0, we can get $h_i(x_0,t)>0$, which completes the proof.
\end{proof}

\begin{lemma}
Each agent can reach the triangular area with local maximum in finite time using Maximal Upgrade Method.
\end{lemma}

\begin{proof}
In Maximal Upgrade Method, before the update, the temperature field remains unchanged. $T_i(x_0)$ is the temperature value of local maximum point in region $D_i$, $T_i(\bf{s}_i(t))$ is the temperature value of each point on the trajectory of agent $i$.
Contruct the following  Lyapunov-Krasovskii function:
$$
V_i(t) = T_i(x_0)-T_i(\bf{s}_i(t)) .
$$
Calculating the time derivative of $V_i(t)$ yields
$$
\begin{aligned}
	\dot{V}_i(t) &=-\bigtriangledown T_i(\bf{s}_i(t))\dot{l}_i^T(t) \\
	&=-\bigtriangledown T_i(\bf{s}_i(t))V_i\frac{\bigtriangledown T_i^T(\bf{s}_i(t))}{||\bigtriangledown T_i(\bf{s}_i(t))||}\\
	&=-V_i||\bigtriangledown T_i(\bf{s}_i(t))||.
\end{aligned}
$$
where $\dot{s}_i(t)$ is defined in (\ref{13}). We can easily get $V_i(t)>0$ and $\dot{V}_i(t)<0$. Hence, the system (\ref{12}) is asymptotical stable. In other words, the agent can reach the local maximum point when $t \longrightarrow \infty$. According to Lyapunov asymptotical stability and the definition of limit, $\forall \epsilon>0$, $\exists T>0$, when $t>T$, one has $||\bf{s}_i(t)-x_0||\leq\epsilon$. Take $\epsilon$ equals the radius of minimum covering circle $r_{min}$, which completes the proof.
\end{proof}
\begin{lemma}
Delaunay triangulation enables each agent to complete workload on its subregion in finite time.
\end{lemma}
\begin{proof}
	Since the use time $T_u^k$ of Real-time Update is fixed, its value does not affect the proof result. In Maximal Update, before updating the temperature field each time, the number of local maximum points corresponding to the current temperature field is limited, which is determined by the limited number of triangles in the triangulation result of the region. For each update of the temperature field, according to Lemma 4.2., $\exists T_1>0$, $||\bf{s}_i(T_1)-x_0||\leq r_{min}$. Then, according to (\ref{18}), the agent can cover the triangular area where the local maximum point is located after stopping the movement, and the coverage time is
$$
T_2=\frac{m^k(x_0)}{p_i(x_0,\bf{s}_i)},
$$
where $m^k(x_0)$ is the workload of local maximum point, $\bf{s}_i$ is the location of agent $i$. It can be concluded that the coverage time of the triangular region where the local maximum point is located is $T_1+T_2$, which is limited. According to the definition of maximum point $x_0$, we have
$$
\bigtriangledown T_i(x_0)=0, \quad H(T_i(x_0))<0.
$$
According to Lemma 4.1, we have $h_i(x_0)>0$. For the triangular region that has not been covered, due to finite triangulation cells of subregion $D_i$, it can be deduced that the coverage time of the region is finite, which completes the proof.
\end{proof}

\begin{theorem}
Algorithm 3 ensures the finite-time completion of cooperative coverage mission.
\end{theorem}
\begin{proof}
When $k<k_{max}$, each iteration takes finte time, which can be proved by Lemma~4.3.. For the $k_{max}$-th itration, the completion time depends on the maximum time for the agent to complete the subregion coverage. It is finite-time and can be proved in a similar way to Lemma~4.3..
\end{proof}
\section{Numerical Simulations}
In this section, we provide simulation results to verify the proposed coverage algorithm and compare it with other existing approaches.
\begin{figure}[t!]
	\centering
	\includegraphics[width=12cm]{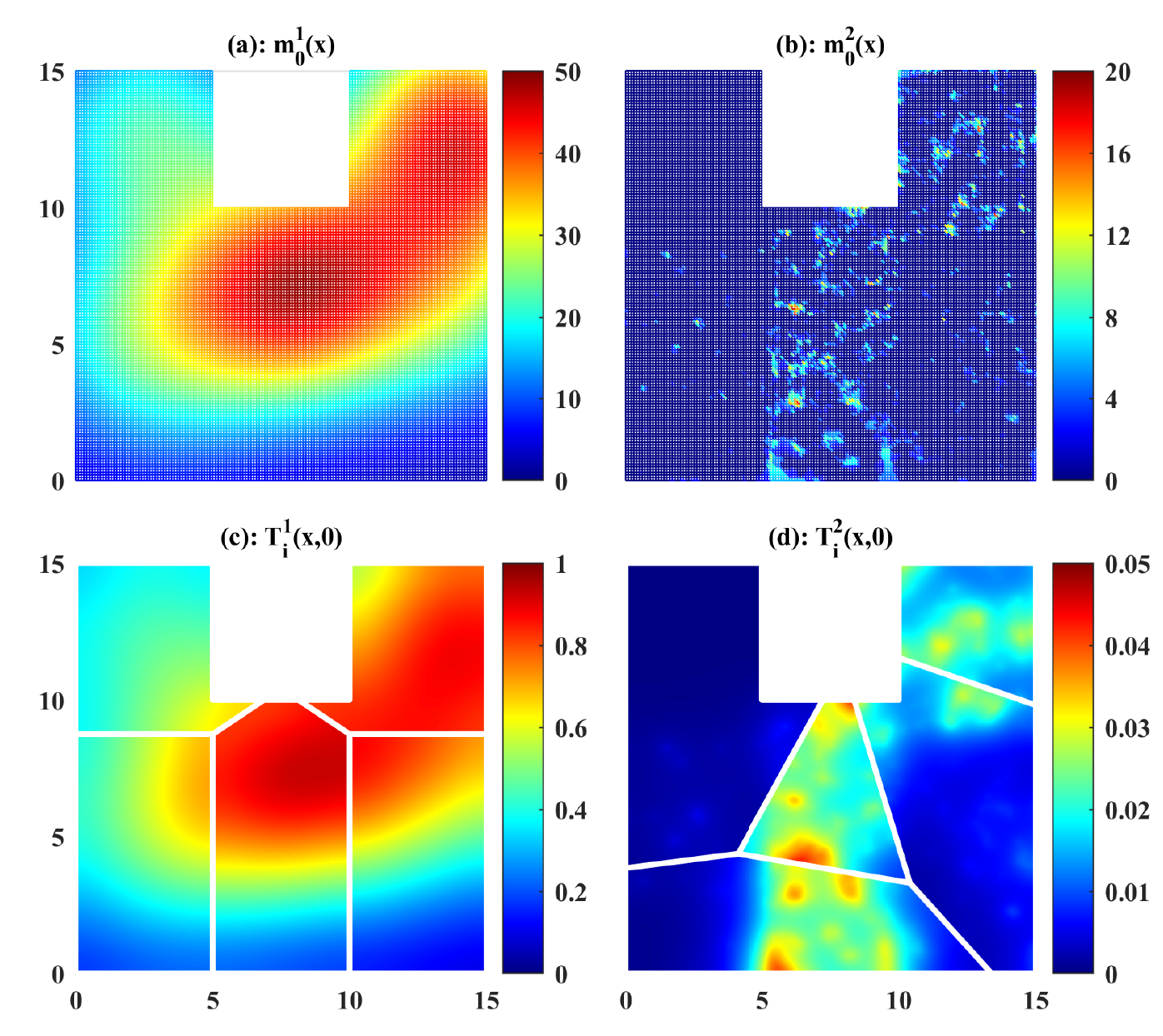}
	\caption{Coverage process of five agents using Algorithm 3 on the region $D$. (a), (b): Initial workload distribution for two iterations. (c), (d): Temprature value distribution at the beginning of first and second iteration.}
\end{figure}
For simplicity, we consider five agents in a region. Region $D$ is a concave area formed by digging $5\times5$ small squares out of $15\times15$ large squares. Parameters are selected as follow: $P=6$, $r=0.5$, $\lambda = 1$, $\alpha = 1$, $\beta=1$, $V_i=0.5$, $T_u=500$, $k_{max}=7$.
Besides, the initial workload on the coverage region is set as follows
\begin{equation}\label{19}
\begin{array}{r}
m_0(\textbf{x})=m_0^1(\textbf{x})=40e^{-\frac{(x-8)^2+(y-7)^2}{20}}+20e^{-\frac{(x-2.5)^2+(y-13)^2}{20}}+40e^{-\frac{(x-14)^2+(y-13)^2}{20}} \\
+20e^{-\frac{(x-2.5)^2+(y-5)^2}{20}}+20e^{-\frac{(x-14)^2+(y-7)^2}{20}}.
\end{array}
\end{equation}
Four snapshots during implementing Algorithm 3 are shown in Figure 3. Through the distribution of temperature value, it is obvious that each region forms its own relatively independent temperature field after Voronoi partition. But at the same time, it also retains the value information of the initial workload distribution. According to the second iteration, although the distribution of workload is discontinuous, it is successfully transformed into a continuous temperature field through the construction of temperature transfer function.
After seven iterations, we finally get the motion trajectories of five agents, as shown in Figure 4 where the square is the starting position and the pentagram is the ending position. Compare Algorithm 3 (ALGO 3) with centralized temperature field control algorithm (\cite{13}) and Algorithm 1 (ALGO 1), and get the curve of workload changing with time in Figure 5.
\begin{figure}[t!]
	\centering
	\begin{minipage}[t]{0.45\linewidth}
		\centering
		\includegraphics[width=7cm]{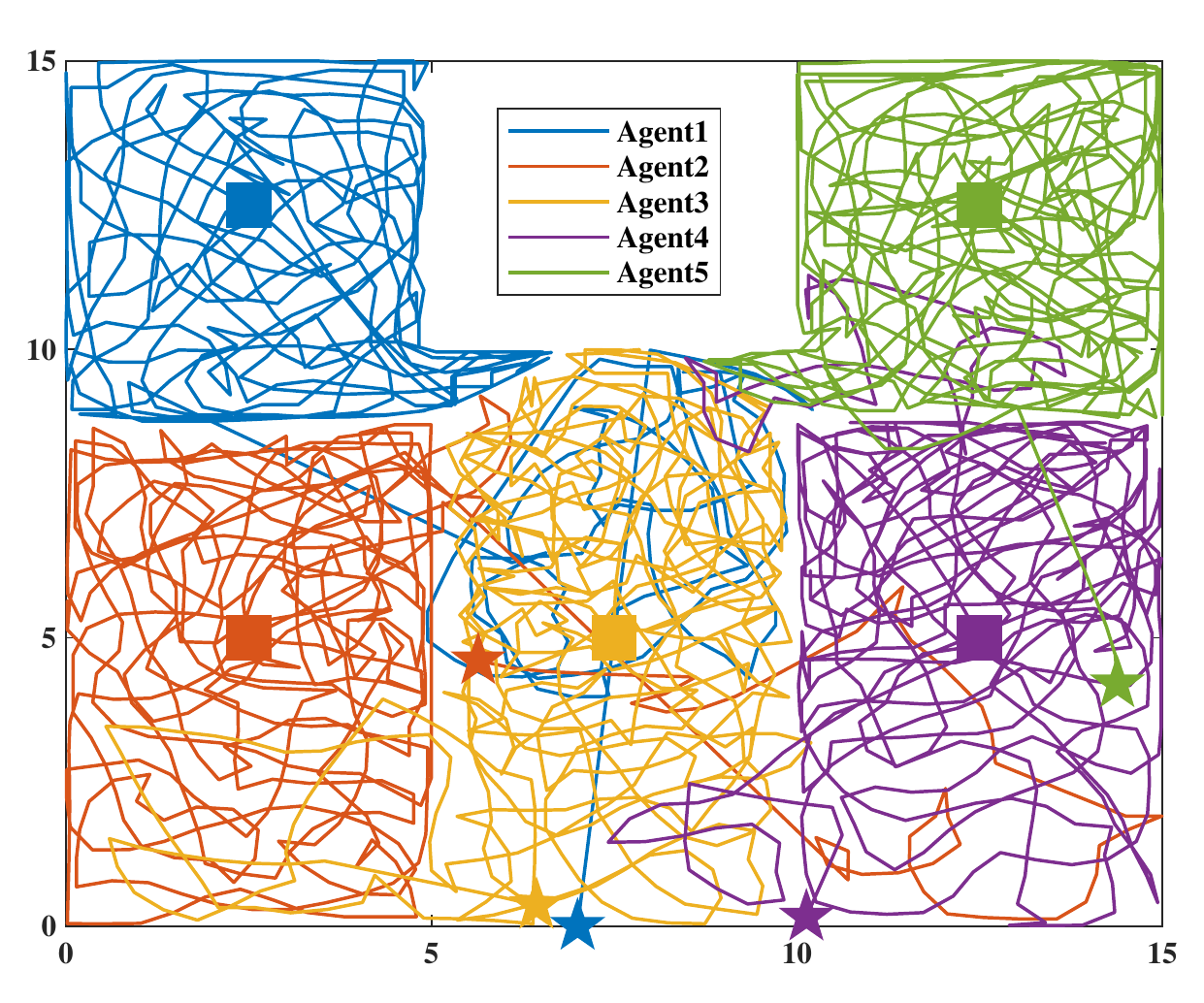}
		\caption{Trajectories of five agents on region $D$.}
	\end{minipage}
	\qquad
	\begin{minipage}[t]{0.45\linewidth}
		\centering
		\includegraphics[width=7.1cm]{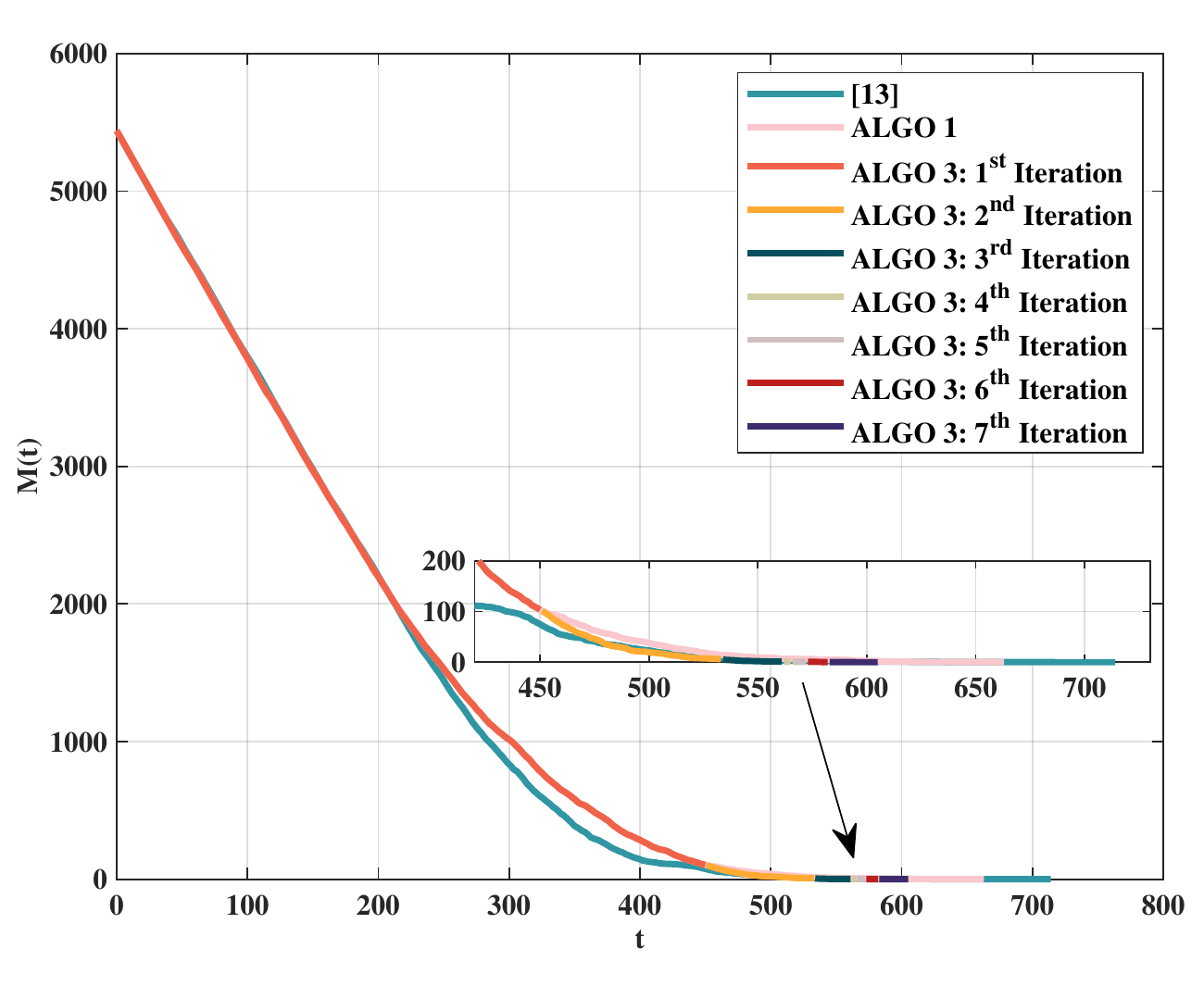}
		\caption{Curve of workload using \cite{13}, ALGO 1 and ALGO 3.}
	\end{minipage}
\end{figure}

During the 1st iteration, the centralized control method achieves higher coverage efficiency. The distributed segmentation method accelerates the reduction of the workload of the sub area, resulting in the phenomenon that the workload of the agent has been emptied through the location, which wastes time. However, it can be seen that the iterative method has brought changes in efficiency in the later stage, and the replacement of work area has alleviated the waste of coverage resources. Based on the optimization objective, we can compare the algorithm effect by observing the difference between the total coverage time and the optimal time in Table 1. It can be seen that Algorithm 3 achieves the best performance.

\begin{table}[t!]
\centering\caption{Comparison of different coverage control algorithms}
\setlength{\tabcolsep}{8mm}{
	\begin{tabular}{cccc}
		\toprule
		Method&T&$\Delta T$& $T^*$\\
		\midrule
		\cite{13}&714&253&\multirow{3}{*}{461}\\
		ALGO 1&663&202&~\\
		ALGO 3&605&144&~\\
		\bottomrule
\end{tabular}}
\end{table}

\section{Conclusions}

This work presents a distributed solution to multi-agent coverage problem in uncertain environment. Voronoi partition is used to divide the whole coverage region into multiple sub-regions, which allows agents to complete the coverage task in a distributed manner. The temperature field induced control strategy enables the agent to move in its own sub-region and complete the coverage task. Besides, the iteration between region partition and coverage control improves the coverage efficiency. Theoretical analysis is conducted to guarantee the finite-time completion of coverage task, numerical simulations are carried out to substantiate the efficiency and robustness of algorithm.

\section*{Acknowledgement}
This work was supported by the Fundamental Research Funds for the Central Universities, China University of Geosciences (Wuhan).

\end{document}